\documentclass{article}%
\usepackage{amscd,amssymb}
\usepackage{amsfonts}
\usepackage{hyperref}
\usepackage{color}
\usepackage{amsmath}
\usepackage{amsfonts}
\usepackage{amssymb}
\usepackage{graphicx}
\usepackage{graphicx,color}%
\newtheorem{theorem}{Theorem}

\newtheorem{corollary}[theorem]{Corollary}

\newtheorem{definition}[theorem]{Definition}
\newtheorem{example}[theorem]{Example}

\newtheorem{lemma}[theorem]{Lemma}

\newtheorem{proposition}[theorem]{Proposition}
\newtheorem{remark}[theorem]{Remark}

\newenvironment{proof}[1][Proof]{\noindent\textbf{#1.} }{\ \rule{0.5em}{0.5em}}
\begin{document}

\title{Invariance Pressure for Control Systems}
\author{Fritz Colonius\\Institut f\"{u}r Mathematik, Universit\"{a}t Augsburg, Augsburg, Germany
\and Alexandre J. Santana and Jo\~{a}o A. N. Cossich\\Departamento de Matem\'{a}tica, Universidade Estadual de Maring\'{a}\\Maring\'{a}, Brazil}
\maketitle

\textbf{Abstract: }Notions of invariance pressure for control systems are
introduced based on weights for the control values. The equivalence is shown
between inner invariance pressure based on spanning sets of controls and on
invariant open covers, respectively. Furthermore, a number of properties of
invariance pressure are derived and it is computed for a class of linear systems.


\section{Introduction\label{Section1}}

This paper extends the notion of invariance entropy for discrete-time and
continuous-time control systems to a notion of invariance pressure and discuss
some of its properties. Invariance entropy (and feedback invariance entropy)
indicates the amount of \textquotedblleft information\textquotedblright%
\ necessary in order to make a subset of the state space invariant, and is
closely related to minimal data rates. Basic references are the seminal paper
Nair, Evans, Mareels and Moran \cite{NEMM04} and the monograph Kawan
\cite{Kawa13}. Further studies of invariance entropy include Da Silva and
Kawan \cite{daSilKawa16a} for hyperbolic control sets, Da Silva \cite{daSil14}
for linear control systems on Lie groups and Colonius, Fukuoka and Santana
\cite{ColoFS13} for topological semigroups.

Invariance entropy is modeled with some analogy to topological entropy of
dynamical systems. A generalization of the latter notion is topological
pressure of dynamical systems where a potential function gives weights to the
points in the state space, cf., e.g., Walters \cite{Walt82}, Viana and
Oliveira \cite{VianO16} or Katok and Hasselblatt \cite{KatH95}. We will
construct a notion of invariance pressure that analogously is based on weights
for the control values.

The main result is the equivalence between the inner invariance pressure based
on spanning sets of controls, and on invariant open covers (see Theorem
\ref{teo8}). Furthermore, a number of properties of invariance pressure are
derived which are analogous to properties of topological pressure for
dynamical systems. Here, however, no full analogy should be expected, since no
notion of separated sets of controls is available. While inner invariance
pressure, as discussed in detail here, is a generalization of inner invariance
entropy, we indicate how also other notions of invariance entropy, in
particular, outer invariance entropy, can be generalized. Furthermore, some
properties of invariance entropy for continuous-time control systems are also
derived and the invariance pressure for a class of linear systems is computed.

The contents of this paper is as follows. Section \ref{Section2} constructs
inner invariance pressure based on spanning sets of controls and on invariant
open covers and shows that they are equivalent. Section \ref{Section3} proves
a number of properties of inner invariance pressure and indicates variants
based on different technical conditions. Finally, Section \ref{Section4}
analyzes invariance pressure for continuous-time control systems and computes
the invariance pressure for a class of linear systems.

\section{Invariance pressure for discrete-time systems\label{Section2}}

In this section we introduce the notion of invariance pressure for
discrete-time control systems. Then a feedback version is defined and it is
shown that these two notions are equivalent.

The considered class of discrete-time control systems have the form%
\begin{equation}
x_{k+1}=F(x_{k},u_{k}),k\in\mathbb{N}_{0}=\{0,1,\ldots\}, \label{2.1}%
\end{equation}
where $F:X\times U\rightarrow X$ and $(X,d)$ is a metric space and $U$ is a
topological space. We assume that $F_{u}:=F(\cdot,u)$ is continuous for every
$u\in U$. Define $\mathcal{U}:=U^{\mathbb{N}_{0}}$ as the set of all sequences
$\omega=(u_{k})_{k\in\mathbb{N}_{0}}$ of elements in the control range $U$. We
endow $\mathcal{U}$ which is the set of control sequences with the product
topology. Sometimes, we will assume that the set of control values $U$ is a
compact metric space, implying that also $\mathcal{U}$ is a compact metrizable
space. The shift $\theta$ on $\mathcal{U}$ is defined by $(\theta\omega
)_{k}=u_{k+1},k\in\mathbb{N}_{0}$. For $x_{0}\in X$ and $\omega\in\mathcal{U}$
the corresponding solution of (\ref{2.1}) will be denoted by%
\[
x_{k}=\varphi(k,x_{0},\omega),k\in\mathbb{N}_{0}.
\]
Where convenient, we also write $\varphi_{k,\omega}(\cdot):=\varphi
(k,\cdot,\omega)$. By induction, one sees that this map is continuous. Observe
also that this is a cocycle associated with the dynamical system on
$\mathcal{U}\times X$ given by%
\[
\Phi(k,\omega,x_{0})=(\theta^{k}\omega,\varphi(k,x_{0},\omega)),k\in
\mathbb{N}_{0},\omega\in\mathcal{U},x_{0}\in X.
\]
We note the following property which is of independent interest (it is not
used in the following).

\begin{proposition}
The shift $\theta$ is continuous and, if $F:X\times U\rightarrow X$ is
continuous, then $\Phi$ is a continuous dynamical system.
\end{proposition}

\begin{proof}
Continuity of $\theta$ follows since the sets of the form%
\[
W=W_{0}\times W_{1}\times\cdots\times W_{N}\times U\times\cdots\subset
U^{\mathbb{N}_{0}}%
\]
with $W_{i}\subset U$ open for all $i$ and $N\in\mathbb{N}$ form a subbasis of
the product topology and the preimages%
\[
\theta^{-1}W=U\times W_{0}\times W_{1}\times\cdots\times W_{N}\times
U\times\cdots
\]
are open. If $F$ is continuous, then induction shows that $\varphi
(k,x_{0},\omega)$ is continuous in $(x_{0},\omega)\in$ $X\times\mathcal{U}$
for all $k$.
\end{proof}

Throughout the text, we will consider a compact set $Q\subset X$ and denote by
$C(U,\mathbb{R})$ the set of all continuous function $f:U\rightarrow
\mathbb{R}$. We suppose that the set $Q$ is strongly invariant in the sense
that for all $x\in Q$ there is $u\in U$ with $F(x,u)\in\mathrm{int}Q$.
Clearly, this means that for all $x\in Q$ there is $\omega\in\mathcal{U}$ with
$\varphi(k,x,\omega)\in\mathrm{int}Q$ for all $k\geq1$. We are interested in
the minimal information to make $Q$ strongly invariant.

\begin{remark}
\label{rem2} At the end of Section \ref{Section3} we will comment on
possibilities to relax the property of strong invariance.
\end{remark}

\subsection{Inner invariance pressure}

The definition of inner invariance pressure will require the following notion
from Kawan \cite[p. 76]{Kawa13}.

\begin{definition}
Let $Q\subset X$ a compact set with nonempty interior and $n\in\mathbb{N}$. We
say that a subset $\mathcal{S}\subset\mathcal{U}$ is a strongly $(n,Q)$%
-\textbf{\textit{spanning set}} if for each $x\in Q$ there is $\omega
\in\mathcal{S}$ such that $\varphi(i,x,\omega)\in\mathrm{int}Q$ for
$i=1,\ldots,n$.
\end{definition}

The minimal cardinality of such a set is denoted by $r_{inv,int}%
(n,Q)\leq\infty$, and \cite[p. 76]{Kawa13} defines the \textit{\textbf{inner
invariance entropy}} of $Q$ by
\[
h_{inv,int}(Q)=\lim_{n\rightarrow\infty}\frac{1}{n}\log r_{inv,int}(n,Q).
\]
In order to construct the inner invariance pressure of control systems let for
$f\in C(U,\mathbb{R})$%
\[
(S_{n}f)(\omega):=\sum_{i=0}^{n-1}f(u_{i}),\ \ \ \omega=(u_{i})_{i\in
\mathbb{N}_{0}}\in\mathcal{U},
\]
and
\[
a_{n}(f,Q):=\inf\left\{  \sum_{\omega\in\mathcal{S}}e^{(S_{n}f)(\omega
)};\ \mathcal{S}\text{ strongly }(n,Q)\text{-spanning}\right\}  .
\]

\begin{definition}
\label{Definition_pressure1}For a discrete-time control system of the form
(\ref{2.1}), a strongly invariant compact set $Q\subset X$ and $f\in
C(U,\mathbb{R})$ consider
\begin{equation}
P_{int}(f,Q)=\lim_{n\rightarrow\infty}\frac{1}{n}\log a_{n}(f,Q). \label{2.2}%
\end{equation}
The \textbf{\textit{inner invariance pressure}} in $Q$ is the map
$P_{int}(\cdot,Q):C(U,\mathbb{R})\rightarrow\mathbb{R}\cup\{-\infty,\infty\}$.
\end{definition}

This definition deserves several comments. First observe that $P_{int}%
(f,Q)\geq0$ for $f\geq0$.

If $f=\mathbf{0}$ is the null function in $C(U,\mathbb{R})$, then
$\sum_{\omega\in\mathcal{S}}e^{(S_{n}\mathbf{0})(\omega)}=\sum_{\omega
\in\mathcal{S}}1=\#\mathcal{S}$, hence
\begin{align}
a_{n}(\mathbf{0},Q)  &  =\inf\left\{  \sum_{\omega\in\mathcal{S}}%
e^{(S_{n}\mathbf{0})(\omega)};\ \mathcal{S}\text{ strongly }%
(n,Q)\text{-spanning}\right\} \nonumber\\
&  =\inf\left\{  \#\mathcal{S};\ \mathcal{S}\text{ strongly }%
(n,Q)\text{-spanning}\right\} \nonumber\\
&  =r_{inv,int}(n,Q). \label{rinv}%
\end{align}
Taking the logarithm, dividing by $n$ and letting $n$ tend to $\infty$ one
finds that $P_{int}(\mathbf{0},Q)=h_{inv,int}(Q)$. Hence the inner invariance
pressure generalizes the inner invariance entropy.

Next we show that it is sufficient to consider finite spanning sets. More
precisely, the following holds.

\begin{proposition}
For a strongly invariant compact set $Q$ and $f\in C(U,\mathbb{R})$ it
suffices to taken in the definition of $a_{n}(f,Q)$ the infimum over all
finite strongly $(n,Q)$-spanning sets.
\end{proposition}

\begin{proof}
First we show for a strongly $(n,Q)$-spanning set $\mathcal{S}$ there exists a
finite strongly $(n,Q)$-spanning set $\mathcal{S}^{\prime}\subset\mathcal{S}$.
In fact, take an arbitrary $x\in Q$. Since $\mathcal{S}$ is strongly
$(n,Q)$-spanning, there is $\omega_{x}\in\mathcal{S}$ with $y_{j}%
:=\varphi(j,x,\omega_{x})\in\mathrm{int}Q$ for $j=1,\ldots,n$. By continuity,
we find open neighborhoods $W_{1},\ldots,W_{n}$ of $x$ such that
$\varphi(j,W_{j},\omega_{x})\subset\mathrm{int}Q$, for all $j=1,\ldots,n$. The
sets $W_{x}=\bigcap_{i=1}^{n}W_{i},x\in Q$, form an open cover of $Q$. By
compactness of $Q$ there are finitely $x_{1},\ldots,x_{k}\in Q$ such that
$Q\subset\bigcup_{i=1}^{k}W_{x_{i}}$. Then $\mathcal{S}^{\prime}%
=\{\omega_{x_{1}},\ldots,\omega_{x_{k}}\}\subset\mathcal{S}$ is strongly
$(n,Q)$-spanning.

To conclude the proof, set
\[
\widetilde{a}_{n}(f,Q)=\inf\left\{  \sum_{\omega\in\mathcal{S}}e^{(S_{n}%
f)(\omega)};\ \mathcal{S}\text{ is a finite strongly }%
(n,Q)\mbox{-spanning set}\right\}  .
\]
It is clear that $a_{n}(f,Q)\leq\widetilde{a}_{n}(f,Q)$. For the reverse
inequality, let $\mathcal{S}$ be strongly $(n,Q)$-spanning. Then, as shown
above, there is a finite strongly $(n,Q)$-spanning subset $\mathcal{S}%
^{\prime}\subset\mathcal{S}$. Hence
\[
\sum_{\omega\in\mathcal{S}^{\prime}}e^{(S_{n}f)(\omega)}\leq\sum_{\omega
\in\mathcal{S}}e^{(S_{n}f)(\omega)},
\]
implying that $\widetilde{a}_{n}(f,Q)\leq a_{n}(f,Q)$ and then equality is proved.
\end{proof}

Based on this result, in the following we will only consider finite spanning
sets. We still have to show that the limit in (\ref{2.2}) actually exists.

\begin{proposition}
\label{limex1} For $f\in C(U,\mathbb{R})$, the following limit exists and
satisfies%
\[
\lim_{n\rightarrow\infty}\frac{1}{n}\log a_{n}(f,Q)=\inf_{n\geq1}\frac{1}%
{n}\log a_{n}(f,Q).
\]

\end{proposition}

\begin{proof}
This follows by a standard lemma in this context (cf., e.g., Walters
\cite[Theorem 4.9]{Walt82} or Kawan \cite[Lemma B.3]{Kawa13}), if we can show
that the sequence $\log a_{n}(f,Q),n\in\mathbb{N}$, is subadditive. Let
$\mathcal{S}_{1}$ be an strongly $(n,Q)$-spanning set and $\mathcal{S}_{2}$ a
strongly $(k,Q)$-spanning set. Then define control sequences of length $n+k$
by
\[
\omega:=(u_{0},\ldots,u_{n-1},v_{0},\ldots,v_{k-1})\in U^{n+k}.
\]
for each $\omega_{1}=(u_{0},\ldots,u_{n-1})\in\mathcal{S}_{1}$ and $\omega
_{2}=(v_{0},\ldots,v_{k-1})\in\mathcal{S}_{2}$. We claim that the set
$\mathcal{S}$ of these control sequences is strongly $(n+k,Q)$-spanning. In
fact, for $x\in Q$ there exist $\omega_{1}\in\mathcal{S}_{1}$ such that
\[
\varphi(j,x,\omega)=\varphi(j,x,\omega_{1})\in\mathrm{int}Q,j=1,\ldots,n.
\]
Since $\varphi(n,x,\omega_{1})\in\mathrm{int}Q\subset Q$ and $\mathcal{S}_{2}$
is strongly $(k,Q)$-spanning, there is a $\omega_{2}\in\mathcal{S}_{2}$ such
that
\[
\varphi(n+j,x,\omega)=\varphi(j,\varphi(n,x,\omega_{1}),\omega_{2}%
)\in\mathrm{int}Q,j=1,\ldots,k.
\]
This shows the claim. Furthermore, for all $\mathcal{S}_{1}$ and
$\mathcal{S}_{2}$%
\begin{equation}
\sum_{\omega\in\mathcal{S}}e^{(S_{n+k}f)(\omega)}=\sum_{\omega\in\mathcal{S}%
}e^{(S_{n}f)(\omega_{1})}e^{(S_{k}f)(\omega_{2})}\leq\sum_{\omega_{1}%
\in\mathcal{S}_{1}}e^{(S_{n}f)(\omega_{1})}\sum_{\omega_{2}\in\mathcal{S}_{2}%
}e^{(S_{k}f)(\omega_{2})}.\nonumber
\end{equation}
Hence $a_{n+k}(f,Q)\leq a_{n}(f,Q)a_{k}(f,Q)$ and the subadditivity property
follows proving the assertion.
\end{proof}

The following example illustrates the definition of invariance pressure in a
simple case.

\begin{example}
Assume that $f\in C(U,\mathbb{R})$ is bounded below (which, naturally, holds,
if $U$ is compact) and that $F(Q,U)\subset\mathrm{int}Q$, that is, the system
always enters the interior of $Q$ when starting in $Q$. We show that
$P_{int}(f,Q)=\inf f$. Since for every strongly $(n,Q)$-spanning set
$\mathcal{S}$ the estimate%
\[
\sum_{\omega\in\mathcal{S}}e^{(S_{n}f)(\omega)}\geq e^{n\inf f}\cdot
\#\mathcal{S\geq}e^{n\inf f}%
\]
holds, it follows that $P_{int}(f,Q)\geq\inf f$. Conversely, our assumption
implies that for $\varepsilon>0$ there exists $u\in U$ with
\[
f(u)\leq\inf f+\varepsilon.
\]
Then the one-point set $\mathcal{S}=\{\omega\}$, where $\omega=(u,u,\ldots)$,
is strongly $(n,Q)$-spanning and%
\[
\sum_{\omega\in\mathcal{S}}e^{(S_{n}f)(\omega)}=e^{(S_{n}f)(\omega)}%
=e^{nf(u)}\leq e^{n\inf f+n\varepsilon}.
\]
Taking the infimum over all strongly $(n,Q)$-spanning sets one finds that the
invariance pressure satisfies%
\[
P_{int}(f,Q)=\lim_{n\rightarrow\infty}\frac{1}{n}\log a_{n}(f,Q)\leq
\lim_{n\rightarrow\infty}\frac{1}{n}\log e^{n\inf f}+\varepsilon=\inf
f+\varepsilon.
\]
Since $\varepsilon>0$ is arbitrary, it follows that $P_{int}(f,Q)\leq\inf f$.
\end{example}

\subsection{Topological feedback pressure}

Next we introduce a notion of invariance pressure based on feedbacks and show
that it coincides with the invariance pressure defined above.

Open covers in entropy theory of dynamical systems are replaced in case of
control systems by invariant open covers, introduced in Nair et al.
\cite{NEMM04}. For control systems of the form (\ref{2.1}) they have the
following form.

\begin{definition}
For a compact subset $Q\subset X$ an invariant open cover $\mathcal{C}%
=(\mathcal{A},\tau,G)$ is given by $\tau\in\mathbb{N}$, a finite open cover
$\mathcal{A}$ of $Q$ and a map $G:\mathcal{A}\rightarrow U^{\tau}$ assigning
to each set $A$ in $\mathcal{A}$ a control function such that $\varphi
(k,A,G(A))\subset Q$ for all $k\in\{1,\ldots,\tau\}$.
\end{definition}

Here $G(A)$ may be considered as a feedback when applied to the elements of
$A$. Let $\mathcal{C}=(\mathcal{A},\tau,G)$ be an invariant open cover. For
any sequence $\alpha=(A_{i})_{i\in\mathbb{N}_{0}}\in\mathcal{A}^{\mathbb{N}%
_{0}}$, we have the control sequence
\[
\omega(\alpha):=(u_{0},u_{1},\ldots)\ \ \ \mbox{ with }(u_{l})_{l=(i-1)\tau
}^{i\tau-1}=G(A_{i-1}),\ \text{for all }i\geq1,
\]
that is,
\[
\omega(\alpha)=(\underbrace{u_{0},\ldots,u_{\tau-1}}_{G(A_{0})},\underbrace
{u_{\tau},\ldots,u_{2\tau-1}}_{G(A_{1})},\ldots).
\]
Then we can define, for each $n\in\mathbb{N}$, the set
\begin{equation}
B_{n}(\alpha):=\{x\in X;\ \varphi(i\tau,x,\omega(\alpha))\in A_{i}\text{ for
}i=0,1,\ldots,n-1\}. \label{B_n}%
\end{equation}
Observe that $B_{n}(\alpha)$ is open in $Q$ and that the control
$\omega(\alpha)$ is uniquely determined by $\alpha$, but not necessarily by
the set $B_{n}(\alpha)$. For each $n\in\mathbb{N}$, letting $\alpha$ run
through all sequences of elements in $\mathcal{A}$, the family
\[
\mathcal{B}_{n}=\mathcal{B}_{n}(\mathcal{C}):=\{B_{n}(\alpha);\ \alpha
\in\mathcal{A}^{\mathbb{N}_{0}}\}
\]
is a finite open cover of $Q$. Here, and in the following, it is used tacitly
that only the first $n$ elements of $\alpha$ are relevant.

We say that a set of controls of the form%
\[
\mathcal{W}_{n}=\{\omega(\alpha_{i});\alpha_{i}\in\mathcal{A}^{\mathbb{N}_{0}%
}\text{ for }i\in I\}
\]
is a generating set of feedback controls (of length $n\tau$) for the invariant
open cover $\mathcal{C}$, if the sets $B_{n}(\alpha_{i}),i\in I$, form a
subcover of $\mathcal{B}_{n}(\mathcal{C})$ which is minimal in the sense that
none of its elements may be omitted in order to cover $Q$. (Its number of
elements needs not be minimal among all subcovers.) Hence $Q=\bigcup_{i\in
I}B_{n}(\alpha_{i})$ and the number of elements $\#I$ in the index set $I$ is
bounded by $\#\mathcal{B}_{n}$.

Define for $\omega=(u_{i})_{i\in\mathbb{N}_{0}}\in\mathcal{U}$%
\[
(S_{n\tau})(\omega)=\sum_{i=0}^{n\tau-1}f(u_{i}),
\]
and set
\[
q_{n}(f,Q,\mathcal{C})=\inf\left\{  \sum_{\omega\in\mathcal{W}_{n}%
}e^{(S_{n\tau}f)(\omega)};\mathcal{W}_{n}\text{ generating for }%
\mathcal{C}\right\}  .
\]

\begin{definition}
Consider a discrete-time control system of the form (\ref{2.1}), a strongly
invariant compact set $Q\subset X$ and $f\in C(U,\mathbb{R})$. For an
invariant open cover $\mathcal{C}=(\mathcal{A},\tau,G)$, put
\begin{equation}
P_{fb}(f,Q,\mathcal{C})=\lim_{n\rightarrow\infty}\frac{1}{n\tau}\log
q_{n}(f,Q,\mathcal{C}) \label{2.3}%
\end{equation}
and
\[
P_{fb}(f,Q)=\inf\{P_{fb}(f,Q,\mathcal{C});\ \mathcal{C}%
\mbox{ is an invariant open cover of }Q\}.
\]
The \textbf{invariance}\textit{\textbf{ feedback pressure}} is the map
$P_{fb}(\cdot,Q):C(U,\mathbb{R})\rightarrow\mathbb{R}\cup\{-\infty,\infty\}.$
\end{definition}

Here are several comments on this definition. If $f=\mathbf{0}$ is the null
function in $C(U,\mathbb{R})$, then
\[
\sum_{\omega\in\mathcal{W}_{n}}e^{(S_{n}\mathbf{0})(\omega)}=\sum_{\omega
\in\mathcal{W}_{n}}1=\#\mathcal{W}_{n},
\]
hence
\begin{align*}
q_{n}(\mathbf{0},Q,\mathcal{C})  &  =\inf\left\{  \sum_{\omega\in
\mathcal{W}_{n}}e^{(S_{n\tau}0)(\omega)};\mathcal{W}_{n}\text{ generating for
}\mathcal{C}\right\} \\
&  =\inf\left\{  \#\mathcal{B};\ \mathcal{B}\text{ a subcover of }%
\mathcal{B}_{n}\right\}  =N(\mathcal{B}_{n};Q),
\end{align*}
where $N(\mathcal{B}_{n};Q)$ denotes the minimal number of elements in a
subcover of $\mathcal{B}_{n}$. Hence one finds that the strong topological
feedback entropy $h_{fb}(\mathcal{C})$ of $\mathcal{C}$ (as defined in Kawan
\cite[p. 70]{Kawa13}) satisfies
\[
h_{fb}(\mathcal{C}):=\lim_{n\rightarrow\infty}\frac{1}{n\tau}\log
N(\mathcal{B}_{n};Q)=\limsup_{n\rightarrow\infty}\frac{1}{n\tau}\log
q_{n}(\mathbf{0},Q,\mathcal{C})=P_{fb}(\mathbf{0},\mathcal{C}),
\]
and so the strong topological feedback entropy of system (\ref{2.1}) satisfies%
\begin{align*}
h_{fb}(Q)  &  :=\inf\{h_{fb}(\mathcal{C});\ \ \mathcal{C}\text{ an invariant
open cover of }Q\}\\
&  =\inf\{P_{fb}(\mathbf{0},\mathcal{C});\ \mathcal{C}\text{ an invariant open
cover of }Q\}=P_{fb}(\mathbf{0},Q).
\end{align*}
Hence the invariance feedback pressure is a generalization of the strong
topological feedback entropy.

The following lemma provides the remaining proof that the limit in (\ref{2.3})
actually exists.

\begin{lemma}
\label{limex2} If $f\in C(U,\mathbb{R})$ and $\mathcal{C}=(\mathcal{A}%
,\tau,G)$ is an invariant open cover of $Q$, then the following limit exists
and satisfies%
\[
\lim_{n\rightarrow\infty}\frac{1}{n}\log q_{n}(f,Q,\mathcal{C})=\inf_{n\geq
1}\frac{1}{n}\log q_{n}(f,Q,\mathcal{C}).
\]

\end{lemma}

\begin{proof}
The assertions will follow from Walters \cite[Theorem 4.9]{Walt82} if the
sequence $\log q_{n}(f,Q,\mathcal{C}),n\in\mathbb{N}$, is subadditive. This
will be shown by constructing a generating set $\mathcal{W}_{n+k}$ from
generating sets $\mathcal{W}_{n}$ and $\mathcal{W}_{k}$ with the desired properties.

Let $\mathcal{W}_{n}=\{\omega(\alpha_{i_{1}}),\ldots,\omega(\alpha_{i_{M}})\}$
and $\mathcal{W}_{k}=\{\omega(\beta_{i_{1}}),\ldots,\omega(\beta_{i_{K}})\}$
be generating sets of feedback controls. Here $\alpha_{i}$ and $\beta_{j}$ are
given by sequences of sets in $\mathcal{A}$ in the form $\alpha_{i}=\left(
A_{\sigma}^{\alpha_{i}}\right)  _{\sigma}$ and $\beta_{j}=\left(  A_{\sigma
}^{\beta_{i}}\right)  _{\sigma}$. Then define for all $i$ and $j$ sequences in
$\mathcal{A}$ by
\[
\alpha_{i}\beta_{j}=\left(  A_{0}^{\alpha_{i}},\ldots,A_{n-1}^{\alpha_{i}%
},A_{0}^{\beta_{j}},\ldots,A_{k-1}^{\beta_{j}},\ldots\right)  .
\]
If we denote by $A_{\sigma}^{\alpha_{i}\beta_{j}}$ the $\sigma$th element of
$\alpha_{i}\beta_{j}$, then%
\[
A_{\sigma}^{\alpha_{i}\beta_{j}}=\left\{
\begin{array}
[c]{rcl}%
A_{\sigma}^{\alpha_{i}}, & \mbox{if} & 0\leq\sigma\leq n-1\\
A_{\sigma-n}^{\beta_{j}}, & \mbox{if} & \sigma\geq n.
\end{array}
\right.
\]
\textbf{Claim:} The set%
\begin{equation}
\left\{  \omega(\alpha_{i}\beta_{j});i\in\{i_{1},\ldots,i_{M}\},j\in
\{j_{1},\ldots,j_{K}\}\right\}  \label{2.4}%
\end{equation}
contains a generating set of feedback controls.

First note that by the cocycle property one finds for $\sigma=0,\ldots,k$%
\[
\varphi_{(\sigma+n)\tau,\omega(\alpha_{i}\beta_{j})}=\varphi_{\sigma
\tau,(\theta^{n\tau}\omega(\alpha_{i}\beta_{j}))}\circ\varphi_{n\tau
,\omega(\alpha_{i}\beta_{j})}=\varphi_{\sigma\tau,\omega(\beta_{j})}%
\circ\varphi_{n\tau,\omega(\alpha_{i})},
\]
and hence%
\[
\varphi_{(\sigma+n)\tau,\omega(\alpha_{i}\beta_{j})}^{-1}=\varphi
_{n\tau,\omega(\alpha_{i})}^{-1}\circ\varphi_{\sigma\tau,\omega(\beta_{j}%
)}^{-1}.
\]
Thus for all $i$ and $j$%
\begin{equation}
B_{n+k}(\alpha_{i}\beta_{j})=B_{n}(\alpha_{i})\cap\varphi_{n\tau,\omega
(\alpha_{i}\beta_{j})}^{-1}B_{k}(\beta_{j}). \label{B_n+k}%
\end{equation}
In fact,%
\begin{align*}
B_{n+k}(\alpha_{i}\beta_{j})  &  =\displaystyle\bigcap_{\sigma=0}%
^{n+k-1}\varphi_{\sigma\tau,\omega(\alpha_{i}\beta_{j})}^{-1}(A_{\sigma
}^{\alpha_{i}\beta_{j}})\\
&  =\displaystyle\bigcap_{\sigma=0}^{n-1}\varphi_{\sigma\tau,\omega(\alpha
_{i}\beta_{j})}^{-1}(A_{\sigma}^{\alpha_{i}\beta_{j}})\cap\varphi
_{n\tau,\omega(\alpha_{i}\beta_{j})}^{-1}\biggl[\ \displaystyle\bigcap
_{\sigma=0}^{k-1}\varphi_{\sigma\tau,\theta^{n\tau}\omega(\alpha_{i}\beta
_{j})}^{-1}(A_{\sigma+n}^{\alpha_{i}\beta_{j}})\biggr]\\
&  =\displaystyle\bigcap_{\sigma=0}^{n-1}\varphi_{\sigma\tau,\omega(\alpha
_{i})}^{-1}(A_{\sigma}^{\alpha_{i}})\cap\varphi_{n\tau,\omega(\alpha_{i}%
\beta_{j})}^{-1}\biggl[\ \displaystyle\bigcap_{\sigma=0}^{k-1}\varphi
_{\sigma\tau,\omega(\beta_{j})}^{-1}(A_{\sigma}^{\beta_{j}})\biggr]\\
&  =B_{n}(\alpha_{i})\cap\varphi_{n\tau,\omega(\alpha_{i}\beta_{j})}^{-1}%
B_{k}(\beta_{j}).
\end{align*}
Clearly the sets $B_{n+k}(\alpha_{i}\beta_{j})$ are elements of $\mathcal{B}%
_{n+k}(\mathcal{C})$. It follows from (\ref{B_n+k}) that they cover $Q$, since
this is valid for the families $\left\{  B_{n}(\alpha_{i});i\in\{i_{1}%
,\ldots,i_{M}\}\right\}  $ and $\left\{  B_{n}(\beta_{j});j\in\{j_{1}%
,\ldots,j_{K}\}\right\}  $. Hence the collection in (\ref{2.4}) is a subcover
of $\mathcal{B}_{n+k}(\mathcal{C})$ and one finds in the family (\ref{2.4}) an
associated generating set of feedback controls which we denote by
$\mathcal{W}_{n+k}$. Thus the \textbf{Claim} is proved.

In order to show subadditivity of the sequence $\log q_{n}(f,Q,\mathcal{C}%
),n\in\mathbb{N}$, note that for all $n,k\in\mathbb{N}$%
\begin{align}
\sum_{\omega\in\mathcal{W}_{n+k}}e^{(S_{(n+k)\tau}f)(\omega)}  &
=\sum_{\omega\in\mathcal{W}_{n+k}}e^{(S_{n\tau}f)(\omega)}e^{(S_{k\tau
}f)(\theta^{n\tau}\omega)}\nonumber\\
&  \leq\sum_{\omega\in\mathcal{W}_{n}}e^{(S_{n\tau}f)(\omega)}\sum_{\omega
\in\mathcal{W}_{k}}e^{(S_{k\tau}f)(\omega)}.\nonumber
\end{align}
Since $\mathcal{W}_{n}$ and $\mathcal{W}_{k}$ are arbitrary it follows that
$q_{n+k}(f,Q,\mathcal{C})\leq q_{n}(f,Q,\mathcal{C})\cdot q_{k}%
(f,Q,\mathcal{C})$. This implies the required subadditivity concluding the proof.
\end{proof}

Next we show that this feedback invariance pressure coincides with the inner
invariance pressure introduced in Definition \ref{Definition_pressure1}. This
generalizes a result for invariance entropy from Colonius, Kawan and Nair
\cite{ColoKN13}.

\begin{theorem}
\label{teo8}If $f\in C(U,\mathbb{R})$ and $Q$ is a strongly invariant compact
subset of $X$, then
\[
P_{int}(f,Q)=P_{fb}(f,Q).
\]

\end{theorem}

\begin{proof}
First we prove the inequality $P_{int}(f,Q)\leq P_{fb}(f,Q)$. Let
$\mathcal{C}=(\mathcal{A},\tau,G)$ be an invariant open cover. Then for
$n\in\mathbb{N}$, every generating set $\mathcal{W}_{n}$ of controls for
$\mathcal{C}$ is a strongly $(n\tau,Q)$-spanning set and hence
\[
a_{n\tau}(f,Q)=\inf_{\mathcal{S}}\sum_{\omega\in\mathcal{S}}e^{(S_{n\tau
}f)(\omega)}\leq\sum_{\omega\in\mathcal{W}_{n}}e^{(S_{n\tau}f)(\omega)},
\]
where the infimum is taken over all strongly $(n\tau,Q)$-spanning set
$\mathcal{S}$. It follows that $a_{n\tau}(f,Q)\leq q_{n}(f,Q,\mathcal{C})$ and
therefore%
\[
P_{int}(f,Q)=\lim_{n\rightarrow\infty}\frac{1}{n\tau}\log a_{n\tau}%
(f,Q)\leq\lim_{n\rightarrow\infty}\frac{1}{n\tau}\log q_{n}(f,Q,\mathcal{C}%
)=P_{fb}(f,Q,\mathcal{C}).
\]
Since this holds for every invariant open cover $\mathcal{C}$, we conclude%
\[
P_{int}(f,Q)\leq\inf_{\mathcal{C}}P_{fb}(f,Q,\mathcal{C})=P_{fb}(f,Q),
\]
where the infimum is taken over all invariant open covers $\mathcal{C}$ of $Q$.

To show that $P_{fb}(f,Q)\leq P_{int}(f,Q)$ we construct an invariant open
cover for $\tau\in\mathbb{N}$. Let $\mathcal{S}$ be a strongly $(\tau
,Q)$-spanning set. For each $\omega\in\mathcal{S}$ consider
\[
A(\omega):=\{x\in Q;\ \varphi(j,x,\omega)\in\mathrm{int}%
Q\mbox{ for }j=1,\ldots,\tau\}.
\]
The set $\mathcal{A}=\{A(\omega);\ \omega\in\mathcal{S}\}$ forms a finite open
cover of $Q$. Now define a map $G:\mathcal{A}\rightarrow U^{\tau}$ by
\[
G(A(\omega))=(\omega_{0},\ldots,\omega_{\tau-1}).
\]
Clearly, $\mathcal{C}:=(\mathcal{A},\tau,G)$ is an invariant open cover of $Q$.

Recall that $\alpha\in\mathcal{A}^{\mathbb{N}_{0}}$ defines a control
$\omega(\alpha)$ and for $n\in\mathbb{N}$ the set $B_{n}(\alpha)$ is given by
(\ref{B_n}),%
\[
B_{n}(\alpha):=\{x\in X;\ \varphi(i\tau,x,\omega(\alpha))\in A_{i}\text{ for
}i=0,1,\ldots,n-1\}.
\]
These sets form on open cover $\mathcal{B}_{n}=\mathcal{B}_{n}(\mathcal{C})$
of $Q$. Consider a generating set of feedback controls of the form%
\[
\mathcal{W}_{n}=\{\omega(\alpha_{i});\alpha_{i}\in\mathcal{A}^{\mathbb{N}_{0}%
}\text{ for }i\in I\},
\]
hence the sets $B_{n}(\alpha_{i}),i\in I$, form a subcover of $\mathcal{B}%
_{n}(\mathcal{C})$ which is minimal. Therefore
\begin{align*}
\sum_{\omega\in\mathcal{W}_{n}}e^{(S_{n\tau}f)(\omega)}  &  =\sum_{\omega
\in\mathcal{B}_{n}}e^{(S_{\tau}f)(\omega)}e^{(S_{\tau}f)(\theta^{\tau}\omega
)}\cdots e^{(S_{\tau}f)(\theta^{(n-1)\tau}\omega)}\\
&  \leq\left(  \sum_{\omega\in\mathcal{B}_{n}}e^{(S_{\tau}f)(\omega)}\right)
\left(  \sum_{\omega\in\mathcal{B}_{n}}e^{(S_{\tau}f)(\theta^{\tau}\omega
)}\right)  \cdots\left(  \sum_{\omega\in\mathcal{B}_{n}}e^{(S_{\tau}%
f)(\theta^{(n-1)\tau}\omega)}\right) \\
&  \leq\left(  \sum_{\omega\in\mathcal{S}}e^{(S_{\tau}f)(\omega)}\right)
^{n}.
\end{align*}
Since the previous inequality holds for all finite strongly $(\tau
,Q)$-spanning sets $\mathcal{S}$, it follows that $q_{n}(f,Q,\mathcal{C}%
)\leq\left[  a_{\tau}(f,Q)\right]  ^{n}$ for all $n\in\mathbb{N}$. Hence
\begin{align*}
P_{fb}(f,Q,\mathcal{C})  &  =\displaystyle\lim_{n\rightarrow\infty}\frac
{1}{n\tau}\log q_{n}(f,Q,\mathcal{C})\leq\displaystyle\lim_{n\rightarrow
\infty}\frac{1}{n\tau}\log\left[  a_{\tau}(f,Q)\right]  ^{n}\\
&  =\frac{1}{\tau}\log a_{\tau}(f,Q).
\end{align*}
Using Proposition \ref{limex1} we conclude that
\[
P_{fb}(f,Q)=\inf_{\mathcal{C}}P_{fb}(f,Q,\mathcal{C})\leq\inf_{\tau
\in\mathbb{N}}\frac{1}{\tau}\log a_{\tau}(f,Q)=P_{int}(f,Q).
\]
\newline
\end{proof}

\section{Properties of the invariance pressure\label{Section3}}

In this section, we collect several properties of invariance pressure which
are analogous to properties of topological pressure for dynamical systems.
Furthermore, we discuss some alternative versions of invariance pressure.

We start with the following technical lemma which will be used in the proof of
Proposition \ref{propr1}.

\begin{lemma}
\label{Lemma_ele}Let $a_{i}\geq0,b_{i}>0,i=1,..,n\in\mathbb{N}$, be real
numbers. Then%
\[
\frac{\sum_{i=1}^{n}a_{i}}{\sum_{i=1}^{n}b_{i}}\geq\min_{i=1,\ldots,n}\left(
\frac{a_{i}}{b_{i}}\right)  .
\]

\end{lemma}

\begin{proof}
Let $n=2$. Then we may assume that $\frac{a_{1}}{b_{1}}\leq\frac{a_{2}}{b_{2}%
}$. Dividing numerator and denominator by $b_{1}$ one can further assume that
$b_{1}=1$, hence the assumption takes the form $a_{1}\leq\frac{a_{2}}{b_{2}}$
and the assertion reduces to $\frac{a_{1}+a_{2}}{1+b_{2}}\geq a_{1}$. This is
equivalent to
\[
a_{1}+a_{2}\geq a_{1}+a_{1}b_{2},\text{ i.e., }a_{2}\geq a_{1}b_{2},
\]
which is our assumption. The induction step from $n$ to $n+1$ follows since%
\[
\frac{\sum_{i=1}^{n+1}a_{i}}{\sum_{i=1}^{n+1}b_{i}}=\frac{\sum_{i=1}^{n}%
a_{i}+a_{n+1}}{\sum_{i=1}^{n}b_{i}+b_{n+1}}\geq\min\left(  \frac{\sum
_{i=1}^{n}a_{i}}{\sum_{i=1}^{n}b_{i}},\frac{a_{n+1}}{b_{n+1}}\right)  \geq
\min_{i=1,\ldots,n+1}\left(  \frac{a_{i}}{b_{i}}\right)  .
\]

\end{proof}

\begin{proposition}
\label{propr1}Consider a discrete-time control system of the form (\ref{2.1}),
let $Q$ be a compact strongly invariant subset and let $f,g\in C(U,\mathbb{R}%
)$ and $c\in\mathbb{R}$. Then the following assertions hold:

(i) if $f\leq g$, then $P_{int}(f,Q)\leq P_{int}(g,Q)$.

(ii) $P_{int}(f+c,Q)=P_{int}(f,Q)+c$.

(iii) If $U$ is compact, then $|P_{int}(f,Q)-P_{int}(g,Q)|\leq\Vert
f-g\Vert_{\infty}$.
\end{proposition}

\begin{proof}
(i) If $f\leq g$, it follows that $\sum_{\omega\in\mathcal{S}}e^{(S_{n}%
f)(\omega)}\leq\sum_{\omega\in\mathcal{S}}e^{(S_{n}g)(\omega)}$ for all
$(n,Q)$-spanning sets $\mathcal{S}$, because the exponential function is
increasing. Hence $a_{n}(f,Q)\leq a_{n}(g,Q)$ and so $P_{int}(f,Q)\leq
P_{int}(g,Q)$.

(ii) One finds that
\begin{align*}
a_{n}(f+c,Q)  &  =\inf\left\{  \sum_{\omega\in\mathcal{S}}e^{(S_{n}%
(f+c))(\omega)};\ \mathcal{S}\text{ }(n,Q)\text{-spanning}\right\} \\
&  =\inf\left\{  e^{nc}\sum_{\omega\in\mathcal{S}}e^{(S_{n}f)(\omega
)};\ \mathcal{S}\text{ }(n,Q)\text{-spanning}\right\} \\
&  =e^{nc}a_{n}(f,Q),
\end{align*}
hence
\begin{align*}
P_{int}(f+c,Q)  &  =\lim_{n\rightarrow\infty}\frac{1}{n}\log a_{n}%
(f+c,Q)=\lim_{n\rightarrow\infty}\frac{1}{n}\log\left(  e^{nc}a_{n}%
(f,Q)\right) \\
&  =c+P_{int}(f,Q).
\end{align*}

(iii) Recall that for $a_{n}(f,Q)$ and $a_{n}(g,Q)$ the infimum is taken over
all strongly $(n,Q)$-spanning sets $\mathcal{S}$. Thus, using Lemma
\ref{Lemma_ele} for the second inequality below, one finds%
\begin{align*}
\frac{a_{n}(g,Q)}{a_{n}(f,Q)}  &  =\frac{\inf_{\mathcal{S}}\left\{
\sum_{\omega\in\mathcal{S}}e^{(S_{n}g)(\omega)}\right\}  }{\inf_{\mathcal{S}%
}\left\{  \sum_{\omega\in\mathcal{S}}e^{(S_{n}f)(\omega)}\right\}  }\geq
\inf_{\mathcal{S}}\left\{  \frac{\sum_{\omega\in\mathcal{S}}e^{(S_{n}%
g)(\omega)}}{\sum_{\omega\in\mathcal{S}}e^{(S_{n}f)(\omega)}}\right\} \\
&  \geq\inf_{\mathcal{S}}\left\{  \min_{\omega\in\mathcal{S}}\frac
{e^{(S_{n}g)(\omega)}}{e^{(S_{n}f)(\omega)}}\right\}  \geq e^{-n\Vert
f-g\Vert_{\infty}}.
\end{align*}
Therefore $\frac{a_{n}(f,Q)}{a_{n}(g,Q)}\leq e^{n\Vert f-g\Vert_{\infty}}$ and
so
\begin{align*}
P_{int}(f,Q)-P_{int}(g,Q)  &  =\lim_{n\rightarrow\infty}\frac{1}{n}\log
\frac{a_{n}(f,Q)}{a_{n}(g,Q)}\leq\lim_{n\rightarrow\infty}\frac{1}{n}\log
e^{n\Vert f-g\Vert_{\infty}}\\
&  =\Vert f-g\Vert_{\infty}.
\end{align*}
Interchanging the roles of $f$ and $g$ one finds assertion (iii).
\end{proof}

Next we discuss changes in the considered set $Q$.

\begin{proposition}
Let $f\in C(U,\mathbb{R})$ and $Q\subset X$ a compact strongly invariant set.
Assume that $Q=\bigcup\nolimits_{i=1}^{N}Q_{i}$ with compact strongly
invariant sets $Q_{1},\ldots,Q_{N}$. Then
\[
P_{int}(f,Q)\leq\max_{1\leq i\leq N}P_{int}(f,Q_{i}).
\]

\end{proposition}

\begin{proof}
For every $i\in\{1,\ldots,N\}$, let $\mathcal{S}_{i}$ a strongly $(n,Q_{i}%
)$-spanning set and define $\mathcal{S}=\bigcup_{i=1}^{N}\mathcal{S}_{i}$.
Then $\mathcal{S}$ is a strongly $(n,Q)$-spanning set with%
\[
\sum_{\omega\in\mathcal{S}}e^{(S_{n}f)(\omega)}\leq\sum_{i=1}^{N}\sum
_{\omega\in\mathcal{S}_{i}}e^{(S_{n}f)(\omega)}.
\]
With
\[
a_{n}(f,Q_{i})=\inf\left\{  \sum_{\omega\in\mathcal{S}_{i}}e^{(S_{n}%
f)(x,\omega)};\ \mathcal{S}_{i}\text{ strongly }(n,Q_{i})\text{-spanning}%
\right\}  ,
\]
we have $a_{n}(f,Q)\leq\sum_{i=1}^{N}a_{n}(f,Q_{i})$. Now Kawan \cite[Lemma
2.1]{Kawa13} implies that%
\begin{align*}
P_{int}(f,Q)  &  =\lim_{n\rightarrow\infty}\frac{1}{n}\log a_{n}%
(f,Q)\leq\limsup_{n\rightarrow\infty}\frac{1}{n}\log\sum_{i=1}^{N}%
a_{n}(f,Q_{i})\\
&  \leq\max_{1\leq i\leq N}\limsup_{n\rightarrow\infty}\frac{1}{n}\log
a_{n}(f,Q_{i})\\
&  =\max_{1\leq i\leq N}P_{int}(f,Q_{i}).
\end{align*}

\end{proof}

Consider two control systems of the form (\ref{2.1}) given by
\begin{equation}
x_{k+1}=F_{1}(x_{k},u_{k})\text{ and }y_{k+1}=F_{2}(y_{k},v_{k}) \label{two}%
\end{equation}
in $X_{1}$ and $X_{2}$ with corresponding solutions $\varphi_{1}%
(n,x,\omega_{1})$ and $\varphi_{2}(n,y,\omega_{2})$ and control spaces
$\mathcal{U}_{1}$ and $\mathcal{U}_{2}$ corresponding to control ranges
$U_{1}$ and $U_{2}$, respectively. Then
\[
z_{k+1}=F(z_{k},w_{k}),
\]
with $z_{k}=(x_{k},y_{k})$, $w_{k}=(u_{k},v_{k})$, $F=(F_{1},F_{2})$, again is
a control system of the form (\ref{2.1}) in $X_{1}\times X_{2}$ with control
space $\mathcal{U}_{1}\times\mathcal{U}_{2}$ and solution $\varphi_{1}%
\times\varphi_{2}:\mathbb{N}_{0}\times(X_{1}\times X_{2})\times(\mathcal{U}%
_{1}\times\mathcal{U}_{2})$,
\[
\left(  \varphi_{1}\times\varphi_{2}\right)  (n,z,\omega)=\left(  \varphi
_{1}\times\varphi_{2}\right)  (n,(x,y),(\omega_{1},\omega_{2}))=(\varphi
_{1}(n,x,\omega_{1}),\varphi_{2}(n,y,\omega_{2})).
\]

\begin{proposition}
Let $f_{i}\in C(U_{i},\mathbb{R})$ and let $Q_{i}\subset X_{i}$ be compact
strongly invariant sets for the control systems in (\ref{two}), $i=1,2$. Then%
\[
P_{int}(f_{1}\times f_{2},Q_{1}\times Q_{2})=P_{int}(f_{1},Q_{1}%
)+P_{int}(f_{2},Q_{2}),
\]
where $f_{1}\times f_{2}\in C(U_{1}\times U_{2},\mathbb{R})$ is defined by
$(f_{1}\times f_{2})(u,v)=f_{1}(u)+f_{2}(v)$.
\end{proposition}

\begin{proof}
Note that $Q_{1}\times Q_{2}\subset$ $X_{1}\times X_{2}$ is a compact strongly
invariant set. Furthermore, if $\mathcal{S}_{i}$ is a strongly $(n,Q_{i}%
)$-spanning set for $Q_{i}$, $i=1,2$, then $\mathcal{S}=\mathcal{S}_{1}%
\times\mathcal{S}_{2}\subset\mathcal{U}_{1}\times\mathcal{U}_{2}$ is a
strongly $(n,Q_{1}\times Q_{2})$-spanning set and%
\begin{align*}
\sum_{\omega\in\mathcal{S}}e^{(S_{n}(f_{1}\times f_{2}))(\omega)}  &
=\sum_{(\omega_{1},\omega_{2})\in\mathcal{S}_{1}\times\mathcal{S}_{2}%
}e^{(S_{n}f_{1}))(\omega_{1})}e^{(S_{n}f_{2}))(\omega_{2})}\\
&  =\sum_{\omega_{1}\in\mathcal{S}_{1}}e^{(S_{n}f_{1}))(\omega_{1})}%
\sum_{\omega_{2}\in\mathcal{S}_{2}}e^{(S_{n}f_{2}))(\omega_{2})}.
\end{align*}
Since $\mathcal{S}_{1}$ and $\mathcal{S}_{2}$ are arbitrary, we obtain%
\[
a_{n}(f_{1}\times f_{2},Q_{1}\times Q_{2})=a_{n}(f_{1},Q_{1})a_{n}(f_{2}%
,Q_{2}).
\]
Therefore%
\begin{align*}
P_{int}(f_{1}\times f_{2},Q_{1}\times Q_{2})  &  =\lim_{n\rightarrow\infty
}\frac{1}{n}\log a_{n}(f_{1}\times f_{2},Q_{1}\times Q_{2})\\
&  =\lim_{n\rightarrow\infty}\frac{1}{n}\log\left[  a_{n}(f_{1},Q_{1}%
)a_{n}(f_{2},Q_{2})\right] \\
&  =P_{int}(f_{1},Q_{1})+P_{int}(f_{2},Q_{2}).
\end{align*}

\end{proof}

Next we show that the inner invariance pressure is invariant under appropriate
conjugacies. Again, consider two control systems as in (\ref{two}). A pair of
maps $(\rho,H)$ is called a skew conjugacy if $\rho:X_{1}\rightarrow X_{2}$
and $H:U_{1}\rightarrow U_{2}$ are homeomorphisms such that%
\begin{equation}
\rho(F_{1}(x,u))=F_{2}(\rho(x),H(u))~\text{for all }x\in X_{1},u\in U_{1}.
\label{conj1}%
\end{equation}
Note that this induces a map $h:\mathcal{U}_{1}\rightarrow\mathcal{U}_{2}$
such that $h(\omega)_{i}=H(\omega_{i})$ for all $i\in\mathbb{N}_{0}$ and the
solutions satisfy%
\begin{equation}
\rho(\varphi_{1}(k,x,\omega))=\varphi_{2}(k,\rho(x),h(\omega))\text{ for all
}n\in\mathbb{N}_{0}. \label{conj2}%
\end{equation}
Clearly, skew conjugacy is an equivalence relation.

\begin{theorem}
Using the above notation, assume that $(\rho,H)$ is a skew conjugacy between
these two systems, and let $f\in C(U_{2},\mathbb{R})$ and suppose that
$Q\subset X_{1}$ is strongly invariant. Then $\rho(Q)$ is strongly invariant
in $X_{2}$ and the inner invariance pressure satisfies%
\[
P_{int}(f\circ H,Q)=P_{int}(f,\rho(Q)).
\]

\end{theorem}

\begin{proof}
The set $\rho(Q)$ is compact by continuity of $\rho$. In order to see that it
is strongly invariant, write $y=\rho(x)\in\rho(Q)$ with $x\in Q$. By strong
invariance of $Q$ there is $u\in U_{1}$ with $F_{1}(x,u)\in\mathrm{int}Q$.
Since $\rho$ is an open map, the conjugacy condition implies for all
$i\in\{1,\ldots,n\}$.%
\[
F_{2}(y,H(u))=F_{2}(\rho(x),H(u))=\rho(F_{1}(x,u))\in\rho(\mathrm{int}%
Q)=\mathrm{int}(\rho(Q)).
\]
If $\mathcal{S}$ is a strongly $(n,Q)$-spanning set, then $h(\mathcal{S})$ is
a strongly $(n,\rho(Q))$-spanning set: In fact, for $y=\rho(x)\in\rho(Q)$
there is $\omega\in\mathcal{S}$ with $\varphi_{1}(i,x,\omega)\in
\mathrm{int}(Q)$, $i=1,\ldots,n$, therefore (\ref{conj2}) implies%
\[
\varphi_{2}(i,y,h(\omega))=\varphi_{2}(i,\rho(x),h(\omega))=\rho(\varphi
_{1}(i,x,\omega))\in\rho(\mathrm{int}(Q))=\mathrm{int}(\rho(Q)).
\]
The same arguments show that for a strongly $(n,\rho(Q))$-spanning set
$\widetilde{\mathcal{S}}$ the set $\mathcal{S}:=h^{-1}(\widetilde{\mathcal{S}%
})$ is strongly $(n,Q)$-spanning. Note also that $(S_{n}f)(h(\omega
))=(S_{n}(f\circ H))(\omega)$. Hence
\[
\sum_{h(\omega)\in h(\mathcal{S})}e^{(S_{n}f)(h(\omega))}=\sum_{\omega
\in\mathcal{S}}e^{(S_{n}f)(h(\omega))}=\sum_{\omega\in\mathcal{S}}%
e^{(S_{n}(f\circ H))(\omega)}%
\]
and it follows that $a_{n}(f,\rho(Q))=a_{n}(f\circ H,Q)$, and $P_{int}(f\circ
H,Q)=P_{int}(f,\rho(Q))$, as claimed.
\end{proof}

Next we prove the power rule for inner invariance pressure. Consider a control
system of the form (\ref{2.1}) with compact strongly invariant set $Q$.
Suppose we take $N\in\mathbb{N}$ steps at once. Then, naturally, the solution
$\varphi(N,x,\omega)$ may be in $\mathrm{int}Q$ while there may exist
$i\in\{1,...,N-1\}$ with $\varphi(i,x,\omega)\not \in Q$. Hence, for a power
rule in invariance problems of discrete-time systems one has to exclude this a-priori.

Starting from control system (\ref{2.1}) define the following control system.
Given $N\in\mathbb{N}$, the control range is $U^{N}=U\times\ldots\times U$ and
the set of corresponding controls is denoted by $\mathcal{U}^{N}$. Then a
bijective relation between the controls in $\mathcal{U}$ and in $\mathcal{U}%
^{N}$ is given by%
\[
i:\mathcal{U}\rightarrow\mathcal{U}^{N}:\omega=(\omega_{k})\mapsto(\omega
_{k}^{N}):=(\omega(Nk),\ldots,\omega(Nk+N-1)).
\]
The solutions will be given by $\varphi^{N}(0,x,\omega)=x$ and for $k\geq1$%
\[
\varphi^{N}(k,x,i(\omega))=\varphi(nN,x,\omega).
\]
Then, these are the solutions of a control system of the form
\begin{equation}
x_{k+1}=F^{(N)}(x_{k},v_{k}),~v_{k}\in U^{N}, \label{msyst}%
\end{equation}
and the solutions can be written as
\[
\varphi^{N}(k,x,\omega)=\varphi_{N,\theta^{N(k-1)}(\omega)}\circ\cdots
\circ\varphi_{N,\omega}(x).
\]
As argued above, in the definition of the strong invariance pressure of system
(\ref{msyst}) we only consider solutions which remain in $Q$ for all times
between the steps of length $N$.

\begin{proposition}
In the above setting we denote by $P_{inv}^{N}(f,Q)$ the inner invariance
pressure of (\ref{msyst}). Then for every $f\in C(U,\mathbb{R})$%
\[
P_{int}^{N}(g,Q)=N\cdot P_{int}(f,Q),
\]
where $g\in C(U^{N},\mathbb{R})$ is given by $g(\omega_{0},\ldots,\omega
_{N-1}):=\sum_{i=0}^{N-1}f(\omega_{i})$.
\end{proposition}

\begin{proof}
If $\mathcal{S}\subset\mathcal{U}$ is a strongly $(nN,Q)$-spanning set for
(\ref{2.1}), then $\mathcal{S}^{N}:=\{i(\omega);\ \omega\in\mathcal{S}\}$ is a
strongly $(n,Q)$-spanning set for (\ref{msyst}). Analogously, if
$\mathcal{S}^{N}$ is a strongly $(n,Q)$-spanning set for (\ref{msyst}), then
$i^{-1}(\mathcal{S}^{N})$ is a strongly $(nN,Q)$-spanning set for (\ref{2.1}).
Therefore
\[
\sum_{\omega\in\mathcal{S}^{N}}e^{(S_{n}g)(\omega)}=\sum_{\omega\in
i^{-1}(\mathcal{S}^{N})}e^{(S_{nN}f)(\omega)}.
\]
We denote
\[
a_{n}^{N}(f,Q):=\inf_{\mathcal{S}^{N}}\left\{  \sum_{\omega\in\mathcal{S}^{N}%
}e^{(S_{n}f)(\omega)}\right\}  ,
\]
where the infimum is taken over all the strongly $(n,Q)$-spanning sets
$\mathcal{S}^{N}$ for (\ref{msyst}). Then $a_{n}^{N}(g,Q)=a_{nN}(f,Q)$ and so
\[
P_{int}^{N}(g,Q)=\lim_{n\rightarrow\infty}\frac{1}{n}\log a_{n}^{N}%
(g,Q)=N\lim_{n\rightarrow\infty}\frac{1}{nN}\log a_{nN}(f,Q)=N\cdot
P_{int}(f,Q).
\]

\end{proof}

The following simple example illustrates inner invariance pressure. A more
elaborate case will be discussed in the next section in the framework of outer
invariance pressure for continuous-time systems.

\begin{example}
\label{Example18}Consider a scalar linear system of the form%
\[
x_{k+1}=ax_{k+1}+u_{k},u_{k}\in U:=[-1,1],
\]
with $a>1$ and let $Q:=\left[  -\frac{1}{a-1}+\varepsilon,\frac{1}%
{a-1}-\varepsilon\right]  $, where $\varepsilon>0$ is small. Let $f\in
C(U,\mathbb{R})$ be given by $f(u)=\left\vert u\right\vert ,u\in\lbrack-1,1]$.
We claim that $P_{int}(f,Q)=\log a=h_{inv,int}(Q)$, where the equality for the
inner invariance entropy of $Q$ has been shown in Colonius, Kawan and Nair
\cite[Example 3.2]{ColoKN13}.

In order to show $P_{int}(f,Q)\geq\log a$, consider for $n\in\mathbb{N}$ a
finite strongly $(n,Q)$-spanning set $\mathcal{S}$. For $\omega\in\mathcal{S}$
define%
\[
Q_{\omega}:=\{x\in Q;\varphi(j,x,\omega)\in\mathrm{int}Q\text{ for }%
j=1,\ldots,n\}.
\]
Then $Q=\bigcup\limits_{\omega\in\mathcal{S}}Q_{\omega}$ and hence the
Lebesgue measure $\lambda$ satisfies $\lambda(Q)\leq\sum\limits_{\omega
\in\mathcal{S}}\lambda(Q_{\omega})$. Furthermore, for $x\in Q_{\omega}$ we
have%
\[
\varphi(n,x,\omega)=a^{n-1}x+\sum_{i=0}^{n-2}a^{i}u_{i}\in Q,
\]
which implies that $\lambda(Q)\geq a^{n}\lambda(Q_{\omega})$. Thus%
\[
\lambda(Q)\leq\sum_{\omega\in\mathcal{S}}\lambda(Q_{\omega})\leq
\#\mathcal{S\cdot}\max_{\omega\in\mathcal{S}}\lambda(Q_{\omega})\leq
\#\mathcal{S}\cdot a^{-(n-1)}\lambda(Q)
\]
and hence $\#\mathcal{S}\geq a^{n-1}$. Since $f(u)\geq0$, it follows that%
\[
a_{n}(f,Q)=\inf\left\{  \sum_{\omega\in\mathcal{S}}e^{(S_{n}f)(\omega
)};\ \mathcal{S}\text{ strongly }(n,Q)\text{-spanning}\right\}  \geq a^{n-1}%
\]
and hence%
\[
P_{int}(f,Q)=\lim_{n\rightarrow\infty}\frac{1}{n}\log a_{n}(f,Q)\geq\log a.
\]
In order to prove $P_{int}(f,Q)\leq\log a$, we use that the inner invariance
entropy is given by $h_{inv,int}(Q)=\log a$. If a solution with $x_{0}\in Q$
and control values $u_{i}\in U$ satisfies for $k\geq1$
\[
\varphi(k,x_{0},\omega)=a^{k-1}x_{0}+\sum_{i=0}^{k-2}a^{i}u_{i}\in
\mathrm{int}Q,
\]
then it follows for every $\delta\in(0,1)$ that $\delta u_{i}\in\delta
U=[-\delta,\delta]\subset\lbrack-1,1]=U$ for all $i$ and%
\[
\delta\varphi(k,x_{0},\omega)=a^{k-1}\delta x_{0}+\sum_{i=0}^{k-2}a^{i}\delta
u_{i}\in\mathrm{int}(\delta Q)\subset\mathrm{int}(Q).
\]
Hence the solution keeps the initial point $\delta x_{0}\in\delta Q$ with
control values $\delta u_{i}\in\delta U$ in $\mathrm{int}(\delta Q)$. Observe
that $f(\delta u_{i})=\left\vert \delta u_{i}\right\vert \leq\delta$.

Take $0<\delta<\frac{1}{a-1}-\varepsilon$. Then for $x_{0}\in Q=\left[
-\frac{1}{a-1}+\varepsilon,\frac{1}{a-1}-\varepsilon\right]  $ there are
$n\in\mathbb{N}$ and $\omega=(u_{i})$ with $u_{i}\in U=[-1,1]$ such that%
\[
\left\vert \varphi(n,x_{0},\omega)\right\vert \leq\delta\text{ and }%
\varphi(k,x_{0},\omega)\in Q\text{ for all }k=1,\ldots,n-1.
\]
This is seen as follows. If $x_{0}\in\left[  0,\frac{1}{a-1}-\varepsilon
\right]  $, we can make a step to the left of $x_{0}$ of length $l$ where
$l\in(0,(a-1)\varepsilon]$ is arbitrary. In fact, using the control value
$u_{0}=-1\in\lbrack-1,1]$ one obtains for $x_{1}=ax_{0}+u_{0}$ that
\[
x_{1}-x_{0}=ax_{0}-x_{0}-1\leq(a-1)\left(  \frac{1}{a-1}-\varepsilon\right)
-1=-(a-1)\varepsilon<0.
\]
Similarly, for $u_{0}=-1+(a-1)\varepsilon\in\lbrack-1,1]$, one computes
$x_{1}=x_{0}$ and hence, by continuity, one can make steps of length $l$ to
the left.

Analogously for $x_{0}\in\left[  \frac{1}{1-a}+\varepsilon,0\right]  $ one can
make steps to the right.

Going several steps, if necessary, one can reach the interval $(-\delta
,\delta)$ from each point of $Q$.

By the arguments above we know that we can stay in the interval $(-\delta
,\delta)$. Together we have shown that there is a time $n_{0}\in\mathbb{N}$
such that for every $x\in Q$ there is a control $\omega$ with $\varphi
(n_{0},x,\omega)\in(-\delta,\delta)$. By continuity, there are finitely many
controls $\omega_{1},\ldots,\omega_{N}$ such that for every $x\in Q$ there is
$\omega_{i}$ with $\varphi(n_{0},x,\omega_{i})\in(-\delta,\delta)$.

Now choose a finite $(n,Q)$-spanning set $\mathcal{S}$ with minimal
cardinality $\#\mathcal{S}=r_{inv,int}(n,Q)$. This yields the set
$\mathcal{S}_{\delta}:=\{\delta\omega;\omega\in\mathcal{S}\}$ of controls with
values in $[-\delta,\delta]$ which keep every element in $\delta Q$.
Concatenations of the controls in $\mathcal{S}_{\delta}$ with the controls
$\omega_{1},\ldots,\omega_{N}$ yields an $(n_{0}+n,Q)$-spanning set
$\mathcal{S}^{\prime}$ with cardinality $\#\mathcal{S}^{\prime}\leq
N\cdot\#\mathcal{S}$. For $k\in\{n_{0}+1,\ldots,n_{0}+n\}$, the controls in
$\mathcal{S}^{\prime}$ have values in $[-\delta,\delta]$, hence
$f(u)=\left\vert u\right\vert \leq\delta$ here.

We compute for $\omega^{\prime}=(u_{i})\in\mathcal{S}^{\prime}$%
\begin{align*}
(S_{n_{0}+n}f)(\omega^{\prime})  &  =\sum_{i=0}^{n_{0}+n-1}f(u_{i})=\sum
_{i=0}^{n_{0}-1}f(u_{i})+\sum_{i=n_{0}}^{n_{0}+n-1}f(u_{i})\\
&  \leq n_{0}\max_{u\in\lbrack-1,1]}\left\vert u\right\vert +n\max
_{u\in\lbrack-\delta,\delta]}\left\vert u\right\vert =n_{0}+n\delta.
\end{align*}
This yields%
\begin{align*}
a_{n+n_{0}}(f,Q)  &  \leq\sum_{\omega^{\prime}\in\mathcal{S}^{\prime}%
}e^{(S_{n+n_{0}}f)(\omega)}\leq\#\mathcal{S}^{\prime}\cdot e^{n_{0}+n\delta
}\leq N\cdot\#\mathcal{S}\cdot e^{n_{0}+n\delta}\\
&  =N\cdot r_{inv,int}(n,Q)\cdot e^{n_{0}+n\delta},
\end{align*}
and hence%
\begin{align*}
P_{int}(f,Q)  &  =\underset{n\rightarrow\infty}{\lim\sup}\frac{1}{n+n_{0}}\log
a_{n+n_{0}}(f,Q)\\
&  \leq\underset{n\rightarrow\infty}{\lim\sup}\left[  \frac{1}{n+n_{0}}\log
N+\frac{n}{n+n_{0}}\frac{1}{n}\log r_{inv,int}(n,Q)+\frac{n_{0}+n\delta
}{n+n_{0}}\right] \\
&  \leq\underset{n\rightarrow\infty}{\lim}\frac{1}{n}\log r_{inv,int}%
(n,Q)+\underset{n\rightarrow\infty}{\lim\sup}\frac{n_{0}+n\delta}{n+n_{0}}.
\end{align*}
Since $\frac{n_{0}+n\delta}{n+n_{0}}\leq2\delta$ for $n$ large enough it
follows that $P_{int}(f,Q)\leq h_{inv,int}(Q)+2\delta$ which implies
$P_{int}(f,Q)\leq h_{inv,int}(Q)$, since $\delta>0$ is arbitrary.
\end{example}

As announced in Remark \ref{rem2}, we conclude this section with some comments
on other versions of invariance pressure that can be constructed in analogy to
versions of invariance entropy, cf. Kawan \cite{Kawa13}.

Call a pair $(K,Q)$ of nonempty subsets of $X$ admissible for control system
(\ref{2.1}), if $K$ is compact and for each $x\in K$ there is $\omega
\in\mathcal{U}$ such that $\varphi(k,x,\omega)\in Q$ for all $k\in
\mathbb{N}_{0}$. Then for $n\in\mathbb{N}$ a subset $\mathcal{S}%
\subset\mathcal{U}$ is called $(n,K,Q)$-spanning if for all $x\in K$ there is
$\omega\in\mathcal{S}$ with $\varphi(k,x,\omega)\in Q$ for $k=0,1,\ldots,n$.
For $f\in C(U,\mathbb{R})$ define%
\[
a_{n}(f,K,Q):=\inf\left\{  \sum_{\omega\in\mathcal{S}}e^{(S_{n}f)(\omega
)};\ \mathcal{S}\text{ }(n,K,Q)\text{-spanning}\right\}  .
\]
Then one can define the invariance pressure as%
\[
P(f,K,Q):=\underset{n\rightarrow\infty}{\lim\sup}\frac{1}{n}\log
a_{n}(f,K,Q).
\]
Another version of invariance pressure can be defined as follows. For
$\varepsilon>0$, the $\varepsilon$-neighborhood $N_{\varepsilon}(Q)$ of
$Q\subset X$ is the set $N_{\varepsilon}(Q):=\{y\in X;$ there is $x\in Q$ with
$d(x,y)<\varepsilon\}$. Given a closed set $Q\subset X$, $\varepsilon>0$ and
$n\in\mathbb{N}$, a set $\mathcal{S}\subset\mathcal{U}$ is called
$(n,Q,N_{\varepsilon}(Q))$-spanning, if for all $x\in Q$ there is $\omega
\in\mathcal{S}$ with $\varphi(k,x,\omega)\in N_{\varepsilon}(Q)$ for all
$k=1,\ldots,n$. For $f\in C(U,\mathbb{R})$ define%
\[
a_{n}(\varepsilon,f,Q):=\inf\left\{  \sum_{\omega\in\mathcal{S}}%
e^{(S_{n}f)(\omega)};\ \mathcal{S}\text{ }(n,Q,N_{\varepsilon}%
(Q))\text{-spanning}\right\}  ,
\]
and%
\[
P(\varepsilon,f,Q):=\underset{n\rightarrow\infty}{\lim\sup}\frac{1}{n}\log
a_{n}(\varepsilon,f,Q).
\]
Then we define the outer invariance pressure as%
\[
P_{out}(f,Q)=\lim_{\varepsilon\rightarrow0}P(\varepsilon,f,Q).
\]
Clearly, $P_{out}(f,Q)=\sup_{\varepsilon>0}P(\varepsilon,f,Q)\leq
P_{int}(f,Q)$.

\section{Invariance pressure of continuous-time systems\label{Section4}}

In this section we discuss invariance pressure for control systems given by
ordinary differential equation and show that it can be characterized using
discretized time. Then we will derive a formula for the outer invariance
pressure of linear control systems.

Throughout we assume that $X$ is a $d$-dimensional smooth manifold,
$U\subset\mathbb{R}^{m}$ is Borel measurable and $\mathcal{U}=\{\omega
:\mathbb{R}\rightarrow U;\ $Lebesgue integrable$\}$. Consider the
continuous-time control system
\begin{equation}
\dot{x}(t)=F(x(t),\omega(t)) \label{4.1}%
\end{equation}
where $F:X\times U\rightarrow TX$ is continuous, $TX$ is the tangent bundle
and for each $u\in\mathbb{R}^{m}$ the map $F_{u}:=F(\cdot,u):X\rightarrow TX$
is a vector field. We assume that $Q\subset X$ is compact and that for all
$x\in Q$ and $\omega\in\mathcal{U}$ a unique solution $\varphi(t,x,\omega)\in
Q,t\geq0$, exists. Furthermore, we assume that $Q$ is controlled invariant,
i.e., for every $x\in Q$ there exists $\omega\in\mathcal{U}$ such that
$\varphi(t,x,\omega)\in Q$ for all $t\geq0$.

In analogy to the discrete-time case, we call a subset $\mathcal{S}%
\subset\mathcal{U}$ a $(\tau,Q)$-spanning set, if $\tau>0$ and for all $x\in
Q$, there exists $\omega\in\mathcal{S}$ such that $\varphi(t,x,\omega)\in Q$
for all $t\in\lbrack0,\tau]$.

For $\tau\geq0$ and $f\in C(U,\mathbb{R})$ define $(S_{\tau}f)(\omega
)=\int_{0}^{\tau}f(\omega(t))dt$ and
\[
a_{\tau}(f,Q):=\inf\left\{  \sum_{\omega\in\mathcal{S}}e^{(S_{\tau}f)(\omega
)};\ \mathcal{S}\text{ }(\tau,Q)\text{-spanning}\right\}  .
\]
The central definition is the following.

\begin{definition}
The invariance pressure in $Q$ of $f\in C(U,\mathbb{R})$ for the control
system (\ref{4.1}) is%
\[
P_{inv}(f,Q)=\limsup_{\tau\rightarrow\infty}\frac{1}{\tau}\log a_{\tau}(f,Q)
\]
and the invariance pressure of (\ref{4.1}) is the map $P_{inv}(\cdot
,Q):C(U,\mathbb{R})\rightarrow\mathbb{R}$.
\end{definition}

The next theorem shows that for the invariance pressure the time may be discretized.

\begin{theorem}
\label{discretiza}If $U$ is compact, then the invariance pressure of system
(\ref{4.1}) satisfies for every $\tau>0$%
\begin{equation}
P_{inv}(f,Q)=\limsup_{n\rightarrow\infty}\frac{1}{n\tau}\log a_{n\tau
}(f,Q)\text{ for all }f\in C(U,\mathbb{R}). \label{4.1b}%
\end{equation}

\end{theorem}

\begin{proof}
For every $f\in C(U,\mathbb{R})$, the inequality%
\[
P_{inv}(g,Q)\geq\limsup_{n\rightarrow\infty}\frac{1}{n\tau}\log a_{n\tau
}(g,Q)
\]
is obvious. For the converse note that the function $g(u):=f(u)-\inf f$ is
nonnegative (if $f\geq0$, it is not necessary to consider the function $g$).
Let $(\tau_{k})_{k\geq1}$, $\tau_{k}\in(0,\infty)$ and $\tau_{k}%
\rightarrow\infty$. Then for every $k\geq1$ there exists $n_{k}\geq1$ such
that $n_{k}\tau\leq\tau_{k}\leq(n_{k}+1)\tau$ and $n_{k}\rightarrow\infty$ for
$k\rightarrow\infty$. Since $g\geq0$ it follows that%
\[
a_{\tau_{k}}(g,Q)\leq a_{(n_{k}+1)\tau}(g,Q)
\]
and consequently
\[
\frac{1}{\tau_{k}}\log a_{\tau_{k}}(g,Q)\leq\frac{1}{n_{k}\tau}\log
a_{(n_{k}+1)\tau}(g,Q).
\]
This yields
\[
\limsup_{k\rightarrow\infty}\frac{1}{\tau_{k}}\log a_{\tau_{k}}(g,Q)\leq
\limsup_{k\rightarrow\infty}\frac{1}{n_{k}\tau}\log a_{(n_{k}+1)\tau}(g,Q).
\]
Since
\[
\frac{1}{n_{k}\tau}\log a_{(n_{k}+1)\tau}(g,Q)=\frac{n_{k}+1}{n_{k}}\frac
{1}{(n_{k}+1)\tau}\log a_{(n_{k}+1)\tau}(g,Q)
\]
and $\frac{n_{k}+1}{n_{k}}\rightarrow1$ for $k\rightarrow\infty$, we obtain
\[
\limsup_{k\rightarrow\infty}\frac{1}{\tau_{k}}\log a_{\tau_{k}}(g,Q)\leq
\limsup_{k\rightarrow\infty}\frac{1}{n_{k}\tau}\log a_{n_{k}\tau}%
(g,Q)\leq\limsup_{n\rightarrow\infty}\frac{1}{n\tau}\log a_{n\tau}(g,Q).
\]
This shows that
\[
P_{inv}(f-\inf f,Q)=\limsup_{n\rightarrow\infty}\frac{1}{n\tau}\log a_{n\tau
}(f-\inf f,Q),
\]
and as in Proposition \ref{propr1} (ii) we have
\begin{align*}
P_{inv}(f,Q)  &  =P_{inv}(f-\inf f,Q)+\inf f=P_{inv}(g,Q)+\inf f\\
&  =\limsup_{n\rightarrow\infty}\frac{1}{n\tau}\log a_{n\tau}(f-\inf f,Q)+\inf
f\\
&  =\limsup_{n\rightarrow\infty}\frac{1}{n\tau}\log e^{-n\inf f}a_{n\tau
}(f,Q)+\inf f\\
&  =\limsup_{n\rightarrow\infty}\frac{1}{n\tau}\log a_{n\tau}(f,Q).
\end{align*}

\end{proof}

The above result can be rephrased in the following form. Define the invariance
pressure at time $1$ of system (\ref{4.1}) by%
\[
P_{inv}^{1}(f,Q)=\limsup_{n\rightarrow\infty}\frac{1}{n}\log a_{n}(f,Q),f\in
C(U,\mathbb{R}),
\]
where
\[
a_{n}(f,Q):=\inf\left\{  \sum_{\omega\in\mathcal{S}}e^{(S_{n}f)(\omega
)};\ \mathcal{S}\text{ }(n,Q)\text{-spanning}\right\}  .
\]

\begin{corollary}
If $U$ is compact, then the invariance pressure of system (\ref{4.1})
satisfies%
\[
P_{inv}(f,Q)=P_{inv}^{1}(f,Q)\text{ for all }f\in C(U,\mathbb{R}).
\]

\end{corollary}

\begin{remark}
Compactness of $U$ has been used in the proof of Theorem \ref{discretiza} only
in order to guarantee that $\inf f>-\infty$ for every $f\in C(U,\mathbb{R})$.
Thus the property in (\ref{4.1b}) holds for arbitrary $U$ if the considered
functions $f$ are bounded below.
\end{remark}

Next we determine the outer invariance pressure for a class of problems with
linear control systems. For a control system of the form (\ref{4.1}) the outer
invariance entropy is defined as follows (cf. Kawan \cite[p. 44]{Kawa13}). The
$\varepsilon$-neighborhood of $Q\subset X$ be denoted by $N_{\varepsilon
}(Q):=\{y\in X;$ there is $x\in Q$ with $d(x,y)<\varepsilon\}$.

Given a closed set $Q\subset X$, $\varepsilon>0$ and $\tau>0$, a set
$\mathcal{S}\subset\mathcal{U}$ is called $(\tau,Q,N_{\varepsilon}%
(Q))$-spanning, if for all $x\in Q$ there is $\omega\in\mathcal{S}$ with
$\varphi(t,x,\omega)\in N_{\varepsilon}(Q)$ for all $t\in\lbrack0,\tau]$.
Denote by $r_{inv}(\tau,\varepsilon,Q)$ denote the minimal number of elements
that a $(\tau,Q,N_{\varepsilon}(Q))$-spanning set can have and%
\begin{equation}
h_{inv}(\varepsilon,Q):=\underset{\tau\rightarrow\infty}{\lim\sup}\frac
{1}{\tau}\log r_{inv}(\tau,\varepsilon,Q)\text{.} \label{outer_ent}%
\end{equation}

\begin{definition}
The outer invariance entropy of a closed subset $Q\subset X$ is defined by%
\[
h_{inv,out}(Q):=\lim_{\varepsilon\rightarrow0}h_{inv}(\varepsilon,Q)\leq
\infty.
\]

\end{definition}

It is obvious that $h_{inv,out}(Q)=\sup_{\varepsilon>0}h_{inv}(\varepsilon
,Q)\leq h_{inv}(Q)$.\medskip

We consider linear control systems of the form%
\begin{equation}
\dot{x}(t)=Ax(t)+Bu(t),~u(t)\in U, \label{4.2}%
\end{equation}
where $A\in\mathbb{R}^{d\times d},B\in\mathbb{R}^{d\times m}$ and
$\varnothing\not =\mathrm{int}U$ with $U\subset\mathbb{R}^{m}$.

The following result is a consequence of Kawan \cite[Theorem 3.1 and its
proof]{Kawa13}.

\begin{theorem}
\label{Kawan3.1}Suppose that $Q\subset\mathbb{R}^{d}$ is a compact controlled
invariant set for system (\ref{4.2}) with $\mathrm{int}Q\not =\varnothing$.
Then%
\[
h_{inv,out}(Q)=\sum_{i=1}^{d}\max(0,\operatorname{Re}\mu_{i}),
\]
where summation is over all eigenvalues $\mu_{i}$ of $A$. Furthermore, the
same result holds if in the definition of the outer invariance entropy the
limit superior in the definition (\ref{outer_ent}) of $h_{inv}(\varepsilon,Q)$
is replaced by the limit inferior.
\end{theorem}

\begin{remark}
The existence of a compact controlled invariant set $Q$ with nonempty interior
can be guaranteed if the matrix pair $(A,B)$ is controllable (i.e.,
$\mathrm{rank}$ $[B,$ $AB,$ $\ldots,$ $A^{d-1}B]$ $=d$) and the matrix $A$ is
hyperbolic (i.e., it has no eigenvalues on the imaginary axis).
\end{remark}

Theorem \ref{Kawan3.1} will be used to prove a theorem on outer invariance
pressure which we define in the following way. For the general system
(\ref{4.1}), $f\in C(U,\mathbb{R})$ and $\varepsilon>0$ let%
\begin{align*}
a_{\tau}(\varepsilon,f,Q)  &  :=\inf\left\{  \sum_{\omega\in\mathcal{S}%
}e^{(S_{\tau}f)(\omega)};\ \mathcal{S}\text{ }(\tau,Q,N_{\varepsilon
}(Q))\text{-spanning}\right\}  ,\\
P_{inv}(\varepsilon,f,Q)  &  :=\underset{\tau\rightarrow\infty}{\lim\sup}%
\frac{1}{\tau}\log a_{\tau}(\varepsilon,f,Q).
\end{align*}

\begin{definition}
For $f\in C(U,\mathbb{R})$ the outer invariance pressure in $Q$ is defined by
$P_{out}(f,Q)=\lim_{\varepsilon\rightarrow0}P_{inv}(\varepsilon,f,Q)$ and the
outer invariance pressure of the control system (\ref{4.1}) is the map
$P_{out}(\cdot,Q):C(U,\mathbb{R})\rightarrow\mathbb{R}$.
\end{definition}

We get the following formula for the outer invariance pressure of linear systems.

\begin{theorem}
\label{Theorem_linear}Consider the linear control system (\ref{4.2}) with
compact convex control range $U$. Let $Q\subset\mathbb{R}^{d}$ be compact and
let $f\in C(U,\mathbb{R})$ be a map such that there are $u_{0}\in U$ and
$x_{0}\in\mathrm{int}Q$ with $f(u_{0})=\min_{u\in U}f(u)$ and $Ax_{0}%
+Bu_{0}=0$ (i.e., $x_{0}$ is an equilibrium for $u_{0}$), and assume that
there is $T_{0}>0$ such that for every $x\in Q$ there are $T\in(0,T_{0}]$ and
$\omega\in\mathcal{U}$ with%
\begin{equation}
\varphi(T,x,\omega)=x_{0}\text{ and }\varphi(t,x,\omega)\in Q\text{ for all
}t\in(0,T]. \label{4.2b}%
\end{equation}
Then the outer invariance pressure is%
\begin{equation}
P_{out}(f,Q)=f(u_{0})+h_{inv,out}(Q)=f(u_{0})+\sum_{i=1}^{d}\max
(0,\operatorname{Re}\mu_{i}), \label{4.3}%
\end{equation}
where summation is over all eigenvalues $\mu_{i}$ of $A$.
\end{theorem}

\begin{proof}
Note that our assumption on $Q$ implies that $Q$ is controlled invariant. Then
the second equality in (\ref{4.3}) is an immediate consequence of Theorem
\ref{Kawan3.1}. We will prove the first equality in (\ref{4.3}) in three steps.

\underline{Step 1:} First we will simplify the assertion. Define
$g(v):=f(u+u_{0})$ on $V:=U-u_{0}$. Then $g(0)=f(u_{0})\leq f(u)=g(u-u_{0})$
for all $u\in U$, hence $g(0)=\min_{v\in V}g(v)$. Consider the control system%
\begin{equation}
\dot{y}(t)=Ay(t)+Bv(t),v(t)\in V. \label{15}%
\end{equation}
A trajectory $\varphi(\cdot,x,\omega)$ of (\ref{4.2}) determines a trajectory
$\psi(\cdot,x-x_{0},\omega-u_{0})$ of (\ref{15}) (here $u_{0}$ is identified
with the corresponding constant control function) and conversely, since%
\begin{align*}
\psi(t,x-x_{0},\omega-u_{0})  &  =e^{At}(x-x_{0})+\int_{0}^{t}e^{A(t-s)}%
B(\omega(s)-u_{0})ds\\
&  =e^{At}x+\int_{0}^{t}e^{A(t-s)}B\omega(s)ds-\left[  e^{At}x_{0}+\int
_{0}^{t}e^{A(t-s)}Bu_{0}ds\right] \\
&  =\varphi(t,x,\omega)-x_{0}.
\end{align*}
Thus $\varphi(t,x,\omega)\in N_{\varepsilon}(Q)$ implies that $\psi
(t,x-x_{0},\omega-u_{0})\in N_{\varepsilon}(Q)-x_{0}=N_{\varepsilon}(Q-x_{0}%
)$. The controllability condition for (\ref{4.2}) implies that for every
$x-x_{0}\in Q-x_{0}$
\[
\psi(T,x-x_{0},\omega-u_{0})=0\text{ and }\psi(t,x-x_{0},\omega-u_{0})\in
Q-x_{0}\text{ for all }t\in\lbrack0,T].
\]
Furthermore, $0\in\mathrm{int}(Q-x_{0})$ since $x_{0}\in\mathrm{int}Q$. It
follows that the $(\tau,Q,\mathrm{int}Q)$-spanning sets $\mathcal{S}$ of
system (\ref{4.2}) give rise to $(\tau,Q-x_{0},\mathrm{int}(Q-x_{0}%
))$-spanning sets $\mathcal{S}-u_{0}$ of system (\ref{15}) and conversely.
Then it follows that the outer invariance pressure $P_{out}(f,Q)$ of system
(\ref{4.2}) coincides with the outer invariance pressure $P_{out}(g,Q-x_{0})$
of system (\ref{15}).

These considerations imply that without loss of generality, we can assume that
$0\in U$ and that $Q\subset\mathbb{R}^{d}$ is a compact set with
$0\in\mathrm{int}Q$ such that for every $x\in Q$ there are $T>0$ and
$\omega\in\mathcal{U}$ with%
\[
\varphi(T,x,\omega)=0\text{ and }\varphi(t,x,\omega)\in Q\text{ for all }%
t\in(0,T]
\]
and that $f\in C(U,\mathbb{R})$ with $f(0)=\min_{u\in U}f(u)$ (we just write
$U$ instead of $U-u_{0}$, $Q$ instead of $Q-x_{0}$ and $f$ instead of $g$).

Then, using the same arguments as in the proof of Proposition \ref{propr1}%
(ii), we find that%
\[
P_{out}(f,Q)=P_{out}(f-f(0),Q)+f(0).
\]
Hence we can further assume without loss of generality that $0=f(0)=\min_{u\in
U}f(u)$. Then the claim takes the form $P_{out}(f,Q)=h_{inv,out}(Q)$.

\underline{Step 2:} Next we show $P_{out}(f,Q)\geq h_{inv,out}(Q)$. Clearly,
it is sufficient to show for all $\varepsilon>0$ that $P_{inv}(\varepsilon
,f,Q)\geq h_{inv}(\varepsilon,Q)$. Using (\ref{4.2b}) together with the fact
that $0$ is an equilibrium, one finds that for every $\tau\geq T_{0}$ and
every $x\in Q$ that there is a control $\omega_{x}$ with $\varphi
(\tau,x,\omega_{x})=0$ and $\varphi(t,x,\omega_{x})\in Q$ for all $t\in
\lbrack0,\tau]$. By uniform continuity in $t\in\lbrack0,\tau]$ there is a
neighborhood of $x$ such that for every $y$ in this neighborhood one has
\[
\varphi(t,y,\omega_{x})\in N_{\varepsilon}(Q)\text{ for all }t\in\lbrack
0,\tau].
\]
Then compactness of $Q$ implies that there is a finite $(\tau,Q,N_{\varepsilon
}(Q))$-spanning set.

Let $\delta>0$. Then for arbitrarily large $\tau$ one finds a finite
$(\tau,Q,N_{\varepsilon}(Q))$-spanning set $\mathcal{S}$ with%
\begin{align*}
P(\varepsilon,f,Q)  &  =\underset{\tau^{\prime}\rightarrow\infty}{\lim\sup
}\frac{1}{\tau^{\prime}}a_{\tau^{\prime}}(\varepsilon,f,Q)\geq\frac{1}{\tau
}\log a_{\tau}(\varepsilon,f,Q)-\delta\\
&  \geq\frac{1}{\tau}\log\sum_{\omega\in\mathcal{S}}e^{(S_{\tau}g)(\omega
)}-2\delta.
\end{align*}
Since $\mathcal{S}$ is $(\tau,Q,N_{\varepsilon}(Q))$-spanning, it follows that
$\#\mathcal{S}\geq r_{inv}(\varepsilon,\tau,Q)$ and, by assumption we also
know that $f(u)\geq f(0)=0$ for all $u\in U$.\ This implies for arbitrarily
large $\tau$, that%
\[
P(\varepsilon,f,Q))\geq\frac{1}{\tau}\#\mathcal{S}-2\delta\geq\frac{1}{\tau
}r_{inv}(\tau,\varepsilon,Q)-2\delta.
\]
For $\tau\rightarrow\infty$ it follows that%
\[
P(\varepsilon,f,Q)\geq\underset{\tau\rightarrow\infty}{\lim\inf}\frac{1}{\tau
}r_{inv}(\tau,\varepsilon,Q)-2\delta.
\]
Since $\delta>0$ is arbitrary, it follows that this inequality also holds for
$\delta=0$. For $\varepsilon\rightarrow0$, this yields%
\[
P_{out}(f,Q)=\lim_{\varepsilon\rightarrow0}P(\varepsilon,f,Q)\geq
\lim_{\varepsilon\rightarrow0}\underset{\tau\rightarrow\infty}{\lim\inf}%
\frac{1}{\tau}r_{inv}(\tau,\varepsilon,Q)=h_{inv,out}(Q).
\]
The last equality follows by the additional property stated in Theorem
\ref{Kawan3.1}.

\underline{Step 3:} Finally we show $P_{out}(f,Q)\leq h_{inv,out}(Q)$. Fix
$\varepsilon>0$. The assertion will follow if we can show that for every
$\delta>0$%
\[
P(\varepsilon,f,Q)\leq h_{inv}(\varepsilon,Q)+\delta.
\]
The strategy will be similar as in Example \ref{Example18}:\ Every point in
$Q$ is steered into a small neighborhood of $0\in\mathbb{R}^{d}$ and kept
there by a spanning set constructed using linearity of the system equation.

Take $\delta>0$. Since $0\in\mathrm{int}Q$ there is $\alpha\in(0,1)$ such that
the $\alpha$-ball $N_{\alpha}(0)$ around $0$ with radius $\alpha$ is contained
in $\mathrm{int}Q$. We may choose $\alpha>0$ small enough such that
$\left\vert u\right\vert <\alpha$ implies $f(u)\leq\delta$. The
variation-of-constants formula shows that for $\beta>0$ every trajectory
$\varphi(t,x_{0},u),t\geq0,$ of system (\ref{4.2}) satisfies%
\[
\beta\varphi(t,x_{0},u)=e^{At}\beta x_{0}+\int_{0}^{t}e^{A(t-s)}B\beta
u(s)ds=\varphi(t,\beta x_{0},\beta u),t\geq0.
\]
Take $\beta<\alpha$ small enough such that $\beta Q\subset N_{\alpha}(0)$ in
$\mathbb{R}^{d}$ and $\beta U\subset N_{\alpha}(0)$ in $\mathbb{R}^{m}$. Then
the controls $\beta u$ take values in $\beta U$ which is a subset of $U$ by
convexity of $U$. Note also that $N_{\alpha}(0)\subset Q$ implies
$N_{\alpha\beta}(0)\subset\beta Q$.

As in Step 2, there is for every $x\in Q$ a control $\omega_{x}\in\mathcal{U}$
with%
\[
\varphi(T_{0},x,\omega_{x})=0\text{ and }\varphi(t,x,\omega_{x})\in Q\text{
for all }t\in(0,T_{0}].
\]
By uniform continuity on $[0,T_{0}]$ one finds for all $y$ in a neighborhood
of $x$ that
\[
\left\Vert \varphi(T_{0},y,\omega_{x})\right\Vert <\alpha\beta\text{ and
}\varphi(t,y,\omega_{x})\in N_{\varepsilon}(Q)\text{ for all }t\in
\lbrack0,T_{0}].
\]
Then compactness of $Q$ implies that there are finitely many controls
$\omega_{1},\ldots,\omega_{N}$\ such that for every $x\in Q$ there is
$\omega_{i}$ with%
\begin{equation}
\left\Vert \varphi(T_{0},x,\omega_{i})\right\Vert <\alpha\beta\text{ and
}\varphi(t,y,\omega_{i})\in N_{\varepsilon}(Q)\text{ for all }t\in
\lbrack0,T_{0}]. \label{19}%
\end{equation}
Thus we have found finitely many controls steering every point in $Q$ into
$N_{\alpha\beta}(0)\subset\beta Q\subset N_{\alpha}(0)\subset\mathrm{int}Q$.
Next we construct controls keeping every point in the ball $N_{\alpha\beta
}(0)$ in the $\varepsilon$-neighborhood of $N_{\varepsilon}(Q)$ (on
arbitrarily large time intervals).

Fix $\tau>0$ and let $\mathcal{S}=\left\{  \omega_{1}^{\prime},\ldots
,\omega_{M}^{\prime}\right\}  $ be a $(\tau,Q,N_{\varepsilon}(Q))$-spanning
set with $\#\mathcal{S}=r_{inv}(\tau,\varepsilon,Q)$. Then it follows that
$\mathcal{S}_{\beta}:=\left\{  \beta\omega_{1}^{\prime},\ldots,\beta\omega
_{M}^{\prime}\right\}  $ is $(\tau,\beta Q,N_{\varepsilon}(Q))$-spanning. The
controls $\beta u$ take values in $\beta U\subset N_{\alpha}(0)\cap U$.
Obviously, $\#\mathcal{S}_{\beta}=M=\#\mathcal{S}=r_{inv}(\tau,\varepsilon,Q)$.

The concatenations of the controls $\omega_{1},\ldots,\omega_{N}$ with the
controls in $\mathcal{S}_{\beta}$ are given for $i=1,\ldots,N$ and
$j=1,\ldots,M$ by%
\[
\omega_{ij}(t):=\left\{
\begin{array}
[c]{ccc}%
\omega_{i}(t) & \text{for} & t\in\lbrack0,T_{0}]\\
\omega_{j}^{\prime}(t-T_{0}) & \text{for} & t>T_{0}%
\end{array}
\right.  .
\]
Now consider $\tau^{\prime}:=\tau+T_{0}$. Then the set%
\[
\mathcal{S}^{\prime}=\left\{  \omega_{ij};~i\in\{1,\ldots,N\}\text{ and }%
j\in\{1,\ldots,M\}\right\}
\]
is $(\tau^{\prime},Q,N_{\varepsilon}(Q))$-spanning. This follows, since
$N_{\alpha\beta}(0)\subset\beta Q$ implies by (\ref{19}) that all points
$\varphi(T_{0},x,\omega_{i})\in\beta Q$. On the interval $[T_{0},\tau^{\prime
}]$ each control only takes values in $\beta U\subset N_{\alpha}(0)$, hence
$f(u)\leq\delta$ here. We have $\#\mathcal{S}^{\prime}=N\cdot M=N\cdot
r_{inv}(\tau,\varepsilon,Q)$ and compute for $\omega_{ij}\in\mathcal{S}%
^{\prime}$%
\begin{align*}
(S_{\tau^{\prime}}f)(\omega_{ij})  &  =\int_{0}^{\tau^{\prime}}f(\omega
_{ij}(\sigma))d\sigma=\int_{0}^{T_{0}}f(\omega_{ij}(\sigma))d\sigma
+\int_{T_{0}}^{\tau^{\prime}}f(\omega_{ij}(\sigma))d\sigma\\
&  \leq T_{0}\max_{u\in U}f(u)+(\tau^{\prime}-T_{0})\max_{\left\vert
u\right\vert \leq\delta}f(u)\leq T_{0}\max_{u\in U}f(u)+\tau\delta.
\end{align*}
This yields%
\begin{align*}
\log a_{\tau^{\prime}}(\varepsilon,f,Q)  &  \leq\log\sum_{\omega_{ij}%
\in\mathcal{S}^{\prime}}e^{(S_{\tau^{\prime}}f)(\omega_{ij})}\\
&  \leq\log\sum_{\omega_{ij}\in\mathcal{S}^{\prime}}e^{T_{0}\max_{u\in
U}f(u)+\tau\delta}\\
&  \leq\log\#\mathcal{S}^{\prime}+T_{0}\max_{u\in U}f(u)+\tau\delta\\
&  \leq\log N+T_{0}\max_{u\in U}f(u)+\tau\delta+\log r_{inv}(\tau
,\varepsilon,Q).
\end{align*}
Note that
\[
\lim_{\tau^{\prime}\rightarrow\infty}\frac{\tau}{\tau^{\prime}}\frac{1}{\tau
}\log r_{inv}(\tau,\varepsilon,Q)=h_{inv}(\varepsilon,Q).
\]
Let $\tau_{k}\rightarrow\infty$ such that for $\tau_{k}^{\prime}=\tau
_{k}+T_{0}$%
\[
P(\varepsilon,f,Q)=\lim_{k\rightarrow\infty}\frac{1}{\tau_{k}^{\prime}}\log
a_{\tau_{k}^{\prime}}(\varepsilon,f,Q).
\]
For $k$ large enough%
\[
\frac{1}{\tau_{k}^{\prime}}\left[  \log N+T_{0}\max_{u\in U}f(u)+\tau
_{k}\delta\right]  \leq\delta,
\]
hence it follows that
\[
P(\varepsilon,f,Q)=\lim_{k\rightarrow\infty}\frac{1}{\tau_{k}^{\prime}}\log
a_{\tau_{k}^{\prime}}(\varepsilon,f,Q)\leq h_{inv}(\varepsilon,Q)+\delta.
\]
Since $\delta>0$ is arbitrary, this implies $P(\varepsilon,f,Q)\leq
h(\varepsilon,Q)$ and the proof is complete.
\end{proof}

\end{document}